\documentclass[11pt, oneside]{article}   	
\usepackage{geometry}                		
\usepackage{graphicx}				
\usepackage{amssymb}
\usepackage{amsthm}

\newcommand{\peven}{p\in2\mathbb{N}}
\newcommand{\pnoteven}{p\notin2\mathbb{N}}

\newcommand\norma[1]{\|#1\|}

\newcommand\innerproduct[2]{\langle #1,#2\rangle}

\newcommand\Rn{\mathbb{R}^n}
\newcommand\sphere[1]{\mathbb{S}^{#1}}
\newcommand{\finiteSet}[3]{\{#1_#2\}_{#2=1}^{#3}}
\newcommand\isometricIn[2]{#1\hookrightarrow #2}
\newcommand\notIsometricIn[2]{#1\not\hookrightarrow #2}

\newtheorem{theorem}{Theorem}
\newtheorem{lemma}{Lemma}
\newtheorem{proposition}{Proposition}
\newtheorem{corollary}{Corollary}
\newtheorem{definition}{Definition}
\newtheorem{example}{Example}

\title{On a question of Pietch}
\author{Yossi Lonke}
\date{}							

\begin{document}
\begin{abstract}
The main result is that a finite dimensional normed space embeds isometrically in $\ell_p$ if and only if it has a discrete Levy $p$-representation.
This provides an alternative answer to a question raised by Pietch, and as a corollary, a simple proof of the fact that unless $p$ is an even integer, the two-dimensional Hilbert space $\ell_2^2$ is not isometric to a subspace of $\ell_p$. The
situation for $\ell_q^2$ with $q\neq 2$ turns out to be much more restrictive. The main result combined with a result of Dor provides a proof of the fact that
if $q\neq 2$ then $\ell_q^2$ is not isometric to a subspace of $\ell_p$ unless $q=p$. Further applications
concerning restrictions on the degree of
 smoothness of finite dimensional subspaces of $\ell_p$ are included as well.
\end{abstract}

\maketitle
\section{Introduction}
\label{intro}
{\sl Which finite-dimensional subspaces of $L_p$  are isometric to subspaces of $\ell_p$}? This question, according to \cite{DJP}, was asked by A.~Pietch. 
The results of \cite{DJP} revealed an answer that depends on whether $p$ is an even integer ($p\in 2\mathbb{N}$) or not. It is summarized by the following theorem.
\begin{theorem}\label{DJPthm}

\begin{enumerate}
\item  If $\peven$, then every finite dimensional subspace of $L_p$ is isometric to a subspace of $\ell_p$. (\cite{DJP}, Theorem~B).
\item If $\pnoteven$, then a closed linear subspace $X$ of $L_p$ is isometric to a subspace of $\ell_p$ if and only if every (equivalently, some) unital subspace (i.e. a subspace containing the constant functions) which is isometric to $X$, consists of functions with discrete distribution. (\cite{DJP}, Theorem~A). 
\end{enumerate}
\end{theorem}

Since every subspace of $L_p$ is isometric to a unital subspace of $L_p$, (\cite{DJP}, Lemma 1.2), it follows that
 {\it if $\pnoteven$ then a subspace $X$ of $L_p$ is isometric to a subspace of $\ell_p$ if and only if
every two-dimensional subspace of $X$ is isometric to a subspace of $\ell_p$}. (\cite{DJP}, Corollary~1.7).

 Since $\ell_2$ is not isomorphic to a subspace of $\ell_p$ unless $p=2$, it follows that (\cite{DJP}, Corollary~1.8):
\begin{proposition}\label{hilbertCase} If $\pnoteven$, then the two-dimensional Hilbert space $\ell_2^2$ is
not isometric to a subspace of $\ell_p$.
\end{proposition}
Proposition~1 was apparently unknown until  A.~Pe{\l}czy{\'n}ski announced the results of~(\cite{DJP}) at the Convex Geometry meeting at Oberwolfach, 1997. A weaker result, with $\ell_p$ replaced by $\ell_p^n$, had been proved earlier by Y.~Lyubich (\cite{Lyub}).
Proposition~1 was mentioned in \cite{kolkon}, and more recently in \cite{Kilbane}.  
Theorem~2 below, which provides an alternative answer to Pietch's question, and captures another aspect of the connection between discrete distributions and isometric embeddings into $\ell_p$,
also provides a direct, short proof of Proposition~1.

\section{Preliminaries}
\label{sec:1}
\subsection{Notation}
Throughout this paper, $p$ denotes a positive, real number. For a measure space $(X,\Sigma,\mu)$ with $\mu$ positive, let $V_p=V_p(X,\Sigma,\mu)$ denote
the vector space of all $\Sigma$-measurable functions such that $\int_{X}|f(x)|^p\,d\mu(x)<\infty$, and $Z_p\subset V_p$ the subspace of all functions vanishing $\mu$-almost everywhere. 
The quotient space $V_p/Z_p$ is denoted by $L_p(X,\Sigma,\mu)$.  If $X=\mathbb{N}$, the set of all subsets of $\mathbb{N}$ is $\Sigma$ and $\mu$ is the counting measure, then the corresponding $L_p$ space is denoted by $\ell_p$. In what follows,
$L_p$ will denote the space $L_p([0,1],\Sigma,\lambda)$, where $\Sigma$ is the  Borel $\sigma$-algebra, and $\lambda$ --- the Lebesgue measure. The standard inner product in $\Rn$ is denoted by $\innerproduct{\cdot}{\cdot}$, and the
Euclidean norm by $\norma{\cdot}_2$. The unit-sphere in $\Rn$ is $\sphere{n-1}=\{u\in\Rn:\norma{u}_2=1\}$, and the normalized spherical measure on $\sphere{n-1}$ is denoted by $\sigma_{n-1}$.  All isometries in this note are between vector spaces, and
are linear, i.e., isometric isomorphisms. All vector spaces are over the field of the real numbers.
\subsection{The Levy $p$-representation}
\begin{definition}\label{levyRepDef}
\begin{enumerate}
\item  A finite dimensional normed space  $(F, \norma{\cdot})$ is said to have a \emph{Levy $p$-representation} 
if there exists a linear isomorphism $J:F\to\Rn$ ({$n=\textrm{dim}\,F$}) and an even measure $\xi$ on $\sphere{n-1}$ such that for every $x\in F$
\begin{equation}\label{levyRep}
\norma{x}^p=\int_{\sphere{n-1}}|\innerproduct{Jx}{v}|^p\,d\xi(v)
\end{equation}
\item The representation (\ref{levyRep}) is said to be \emph{discrete} if its support is discrete, i.e if
 there exists a sequence of positive numbers $\finiteSet{a}{j}{\infty}$ and a sequence
of points  $\finiteSet{v}{j}{\infty}$ in~$\sphere{n-1}$ such that 
$\xi=\sum_{j=1}^{\infty}a_j(\delta_{v_j}+\delta_{-v_j})$, where $\delta_v$ denotes the unit-mass measure concentrated at~${v\in\sphere{n-1}}$.

\end{enumerate}
\end{definition}

It is well known that every finite dimensional subspace of $L_p$ has a Levy $p$-representation. For a proof, see (\cite{kolBook}, Lemma~6.4).
Moreover, if $\pnoteven$, then the even measure $\xi$ in (\ref{levyRep}) is uniquely determined by
the norm, because in that case, the linear span of the functions $|\innerproduct{y}{\cdot}|^p$, with $y\in\Rn$, is dense in $C_e(\sphere{n-1})$, the Banach space of even, continuous functions on $\sphere{n-1}$. 
(\cite{Neyman}, Th. 2.1). \footnote {However, the uniqueness of $\xi$ in (\ref{levyRep}) for $\pnoteven$ was proved earlier in (\cite{Kanter}, Corollary~1)).}
If $\peven$, then the linear span of the functions $|\innerproduct{y}{\cdot}|^p$ coincides with the set of homogeneous polynomials of degree~$p$, a set not large enough
to uniquely determine the measure in (\ref{levyRep}). 

Different choices of an isomorphism $J$ in Definition~\ref{levyRepDef} lead to different measures in
the Levy $p$-representations. Nevertheless, a simple change-of-density argument shows that the property of a f.d. space, of having a discrete Levy $p$-representation, (i.e.,
a representation (\ref{levyRep}) with a discrete measure $\xi$), 
is invariant under isometries of~$F$. Note that in general this does not imply  that having one discrete Levy $p$-representation forces all Levy $p$-representations to be discrete. 
If $\peven$, then a finite dimensional subspace of $L_p$ may have a discrete Levy $p$-representation as well as a non-discrete one. In fact, if $\peven$ then \emph{every} f.d subspace of $L_p$ \emph{must have} a discrete Levy $p$-representation. 
This is Corollary~\ref{evenCase} below. However, it is easy to construct concrete examples, the simplest of which is a discrete Levy $2$-representation
of $\ell_2^n$.
\begin{example}
For every $p>0$ and $n\geq 1$, the invariance of $\sigma_{n-1}$ under orthogonal transformations shows that for every ${u\in\Rn}$
\begin{equation}\label{el2Rep}
\norma{u}_2^p=c_{p,n}\int_{\sphere{n-1}}|\innerproduct{u}{v}|^p\,d\sigma_{n-1}(v)
\end{equation}
where $c_{p,n}=\int_{\sphere{n-1}}|v_1|^p\,d\sigma_{n-1}(v)$, with $v_1$ denoting the first coordinate of $v$ with respect to the standard basis. Take $p=2$, and let $\finiteSet{v}{i}{n}$ be any orthonormal basis
in~$\Rn$. Let $\mu$ be the discrete measure on $\sphere{n-1} $ concentrated at $\{\pm v_i\}_{i=1}^n$, assigning mass $\frac{1}{2}$ to each one of the $2n$ points. Then for every $u\in\Rn$
$$\norma{u}_2^2=\sum_{i=1}^n|\innerproduct{u}{v_i}|^2=\int_{\sphere{n-1}}|\innerproduct{u}{v}|^2\,d\mu(v)$$
Thus the discrete measure $\mu$ and the invariant measure $c_{2,n}\sigma_{n-1}$ both provide a Levy $2$-representation of $\ell_2^n$.

\end{example}

\section{An alternative answer to Pietch's question}
\label{sec:2}
\begin{theorem}\label{mainTheorem}
A finite dimensional normed space  is isometric to a subspace of $\ell_p$ if and only if
it has a discrete Levy $p$-representation.
\end{theorem}
\begin{proof}
Fix an $n$-dimensional normed space $F$. If $F$ has a discrete Levy $p$-representation, then there exists an isomorphism $J:F\to\Rn$, a sequence of points $\{v_j\}_{j=1}^{\infty}$ in $\sphere{n-1}$ and a sequence of positive numbers $\finiteSet{a}{j}{\infty}$ such that  for every $x\in F$
$$\norma{x}^p=\int_{\sphere{n-1}}|\innerproduct{J(x)}{v}|^p\,d\mu(v)=\sum_{j=1}^{\infty}a_j|\innerproduct{J(x)}{v_j}|^p$$
The map $T:F\to\ell_p$ defined by 
$$T(x):=\sum_{j=1}^{\infty}a_j^{1/p}\innerproduct{J(x)}{v_j}e_j,\qquad (\forall x\in F)$$
is an isometry from $F$ into $\ell_p$. 

Conversely, assume there exists an isometry $T:F\to\ell_p$. Let ${J:F\to\Rn}$
be an isomorphism. Put $K=J(B_F)$, where $B_F$ is the unit-ball of $F$. The map $S=T\circ J^{-1}$ is then an isometry from $(\Rn,\norma{\cdot}_K)$ into $\ell_p$.
Let $f_1,\dots, f_n$ denote the standard basis in~$\Rn$, and $\finiteSet{e}{j}{\infty}$ the standard basis in $\ell_p$.
For each $1\leq k\leq n$, there exists a sequence ${\{x_{k,j}\}_{j=1}^{\infty}\in\ell_p}$ such that $S(f_k)=\sum_{j=1}^{\infty}x_{k,j}e_j$.
Put ${v_j=(x_{1,j},x_{2,j},\dots , x_{n,j})}$. Then for every ${u\in\Rn}$, 
$$S(u)=S\left(\sum_{k=1}^n\innerproduct{u}{f_k}f_k\right)=\sum_{k=1}^n\innerproduct{u}{f_k}\sum_{j=1}^{\infty}x_{k,j}e_j=\sum_{j=1}^{\infty}\innerproduct{u}{v_j}e_j$$
Hence, since $S$ is a linear isometry,
\begin{equation}\label{isometry}
\norma{u}_K^p=\norma{S(u)}_p^p=\sum_j|\innerproduct{u}{v_j}|^p\qquad (\forall u\in\Rn)
\end{equation}
Put $I=\{j:v_j\neq 0\}$. Define an even measure $\xi$ on $\sphere{n-1}$ that assigns mass $\frac{1}{2}\norma{v_j}^p_2$ to the points $\pm \frac{v_j}{\norma{v_j}_2}$ if $j\in I$, and zero otherwise. Then (\ref{isometry}) can be rewritten as
$$\norma{u}_K^p=\int_{\sphere{n-1}}|\innerproduct{u}{v}|^p\,d\xi(v)\qquad (\forall u\in\Rn)$$
Therefore, for every $x\in F$, as $J:F\to (\Rn,\norma{\cdot}_K)$ is an isometry,
$$\norma{x}^p=\norma{J(x)}_K^p=\int_{\sphere{n-1}}|\innerproduct{J(x)}{v}|^p\,d\xi(v)$$
Hence $F$ has a discrete Levy $p$-representation.\qed
\end{proof}
{\bf Remark 1} Theorem~\ref{mainTheorem} 
 provides a simple answer to Pietch's question in terms of the measures appearing in the Levy $p$-representation of finite dimensional subspaces of~$L_p$.
It should be noted that the deep part of Theorem 1 deals with infinite-dimensional subspaces of $L_p$, of which Theorem~2 remains silent. However, Theorem~2 provides a quick way to prove Proposition~1,
which follows as an immediate corollary, because for $\pnoteven$ the unique Levy $p$-representation of the finite-dimensional Euclidean norm is not discrete.

Another corollary, already mentioned above, follows immediately from Theorem~\ref{DJPthm}, part~1 and Theorem~\ref{mainTheorem}.
\begin{corollary}\label{evenCase}
If $\peven$, then every finite dimensional subspace of $L_p$ has a discrete Levy $p$-representation.
\end{corollary}

\section{Applications}
\label{sec:3}
Throughout this section, it will be convenient to denote the statement "\emph{$X$ isometrically embeds in $Y$}" by the notation $\isometricIn{X}{Y}$, and the negation of the same statement by $\notIsometricIn{X}{Y}$.

In \cite{Dor}, Dor gave a complete answer to the question --- \emph{for which $p,q$ is it true that  $\isometricIn{\ell_q^n}{L_p}$}? Here is his answer (\cite{Dor}, Theorem 2.1).
\begin{theorem}\label{DorTheorem}(Dor's Theorem)
Assume $n\geq 2$, $p\geq 1$  and $q\geq 1$. Then $\isometricIn{\ell_q^n}{L_p}$ only in one of the following situations:
\begin{itemize}
\item[(a)] $p<q\leq 2$.
\item[(b)] $q=2$.
\item[(c)] $p=q$.
\item[(d)] $n=2, p=1$, and any $q\in[1,\infty]$.
\end{itemize}
\end{theorem}
As an application of Theorem 2 above, we have an "$\ell_p$ version" of Dor's theorem, which turns out to be
much more restrictive.
\begin{theorem}\label{application}
Assume $n\geq 2$, $p\geq 1$  and $q\geq 1$. 
\begin{itemize}
\item[(a)]  If $q\neq 2$, then $\notIsometricIn{\ell_q^n}{\ell_p}$ unless $p=q$.
\item[(b)] $\notIsometricIn{\ell_2^n}{\ell_p}$ unless $p\in 2\mathbb{N}$. 
\end{itemize}
\end{theorem}
The case of $p=1$ is rather special, because there is much more information regarding the Levy $1$-representation
compared with Levy $p$-representations where $p>1$. For $p=1$, a striking difference between Dor's Theorem and Theorem~\ref{application} emerges: while every two-dimensional normed space
embeds isometrically into $L_1$, many two-dimensional norms do not embed isometrically in $\ell_1$. We will need the next Lemma in the proof
of Theorem~\ref{application}.
 \begin{lemma}\label{theLemma}
Let $\norma{\cdot}$ be a norm on $\mathbb{R}^2$ such that the function
$$G(\theta)=\norma{(\cos\theta,\sin\theta)},\quad 0\leq\theta\leq 2\pi$$
is of class $C^2$. Then $(\mathbb{R}^2,\norma{\cdot})$ is not isometric to a subspace of $\ell_1$.
\end{lemma}
\begin{proof}
We can write for every $0\leq\varphi\leq 2\pi$,
\begin{equation}\label{1Levy}
G(\varphi)=\frac{1}{4}\int_0^{2\pi}|\cos(\varphi-\theta)|\left[G(\theta-\pi/2)+G''(\theta-\pi/2)\right]\,d\theta,
\end{equation}
as can be verified by partial integration. (\cite{Schneider}, p.195).  Observe that (\ref{1Levy}) is
nothing but the unique Levy $1$-representation of $(\mathbb{R}^2,\norma{\cdot})$ (with respect to a fixed basis), with an even measure on $\mathbb{S}^1$ whose density with respect to $d\theta/4$ is the continuous, nonnegative function ${(G+G'')(\theta-\pi/2)}$.
Hence if $G$ is of class $C^2$, then  $(\mathbb{R}^2,\norma{\cdot})$ does not have a discrete Levy $1$-representation. By Theorem~2, it is does not embed isometrically in~$\ell_1$.

\end{proof}
\begin{proof} {\it of Theorem~\ref{application}}.

Part (b) is not new, and is added for the sake of completeness. It is an immediate consequence of part~$1$ of Theorem~\ref{DJPthm} and Proposition~\ref{hilbertCase}. 
Let us prove part~(a). 

Assume $q\neq 2$. The cases $q>2$ and $q<2$ are treated separately.
\begin{itemize}
\item[(i)] $q>2$. Assume $\isometricIn{\ell_q^2}{\ell_p}$. Since $\isometricIn{\ell_p}{L_p}$, Dor's theorem implies that ${p=1}$, or ${p=q}$. But ${p=1}$ contradicts Lemma~\ref{theLemma},
because the $\ell_q^2$-norm is of class $C^2$ for $q>2$.
Therefore, if $q>2$ then $\notIsometricIn{\ell_q^n}{\ell_p}$ unless $p=q$.
\item[(ii)] $q<2$. Excluding the trivial case $p=q$, Dor's theorem implies we need only consider $1\leq p < q$. It is well known that for $1\leq p<q<2$, the space $\ell_q^n$ has a Levy $p$-representation obtained by
means of the standard $n$-dimensional $q$-stable measure $\mu_{q}$, whose density is the Fourier transform of the function $\exp(-\norma{x}_q^q)$. In fact, one has for every $x\in\mathbb{R}^n$,
\begin{equation}\label{qstable}
\norma{x}_q^p=\frac{1}{\int_{\mathbb{R}}|t|^p\,d\nu_q(t)}\int_{\mathbb{R}^n}|\innerproduct{x}{\xi}|^p\,d\mu_q(\xi)
\end{equation}
where $\nu_q$ is the $1$-dimensional $q$-stable measure. That the one-dimensional integral appearing in (\ref{qstable}) is finite for $1\leq p<q<2$ is a classical fact. An exact calculation
of that integral appeared in \cite{kol91} (ibid. p. 762;  an alternative proof, without calculating the exact value, can be found in \cite{Woj}, Proposition 15, p.~93). 
A Levy $p$-representation of $\ell_q^n$ (in the sense of definition (\ref{levyRepDef}) above)  is readily derived from (\ref{qstable}) by
projecting  $\mu_q$ from $\mathbb{R}^n$ onto the sphere $\mathbb{S}^{n-1}$. More precisely, for each Borel
subset $B\subset\mathbb{S}^{n-1}$, let $\tilde{B}\subset\mathbb{R}^n$ denote the union of all lines in $\mathbb{R}^n$ passing through~$B$, that is, $\tilde{B}=\{tu: t\in\mathbb{R}, u\in B\}$, and define
\begin{equation}\label{pProj}
\nu_q(B)=\frac{1}{2}\int_{\tilde{B}}\norma{x}_2^p\,d\mu_q(x),
\end{equation}
where $\norma{x}_2$ is the Euclidean norm of $x$.
Then one has
\begin{equation}\label{pqLevy}
\int_{\mathbb{R}^n}|\innerproduct{x}{\xi}|^p\,d\mu_q(\xi)=\int_{\mathbb{S}^{n-1}}|\innerproduct{x}{u}|^p\,d\nu_q(u),
\end{equation}
which, combined with (\ref{qstable}) yields a Levy $p$-representation of the space $\ell_q^n$ . In \cite{kol91}, the measure $\nu_q$ in (\ref{pProj}) is called the \emph{$p$-projection of $\mu_q$}. (ibid. p.~760).
  Since the standard $n$-dimensional $q$-stable measure $\mu_q$ has
a smooth density, namely, the Fourier transform of $\exp(-\norma{x}_q^q)$, the integral in~(\ref{pProj}) must vanish along every one-dimensional line in $\mathbb{R}^n$, and in particular, its $p$-projection $\nu_q$ cannot be discrete; for if $u_0\in\mathbb{S}^{n-1}$ is a point to which $\nu_q$ assigns positive mass, then the integral in~(\ref{pProj}), 
taken over the line ${\{tu_0:t\in\mathbb{R}\}}$ would be positive. Since $1\leq  p<2$, the Levy $p$-representation of the $\ell_q^n$-norm is unique. Hence the $\ell_q^n$-norm does not have a discrete Levy $p$-representation. It follows from Theorem~\ref{mainTheorem} that
 $\notIsometricIn{\ell_q^n}{\ell_p}$. 
\end{itemize}
\bigskip
Summarizing the above, we find that if $q\neq 2$ then $\notIsometricIn{\ell_q^n}{\ell_p}$ unless $p=q$, and $\notIsometricIn{\ell_2^n}{\ell_p}$ unless $p$ is an even integer. This completes the proof of Theorem~\ref{application}.

\end{proof}
Utilizing the main theme of this section we can prove an additional result, which places restrictions on the smoothness of finite dimensional subspaces of $\ell_p$.
\begin{proposition}
Fix a number $p\geq 1$, $p\notin 2\mathbb{N}$. Let $\norma{\cdot}$ be a norm in $\mathbb{R}^n$ whose restriction to $\mathbb{S}^{n-1}$ is of class $C^{2r}$, where $r\in\mathbb{N}$ and $2r>n+p$. Then 
$(\mathbb{R}^n,\norma{\cdot})$ does not embed isometrically in $\ell_p$.
\end{proposition}
\begin{proof} Since $\min_{u\in\mathbb{S}^{n-1}}\norma{u}>0$, the function $x\to\norma{x}^p$ is also of class $C^{2r}$ on~$\sphere{n-1}$. Our proposition is then an immediate corollary of a result by Koldobsky,
 asserting that under the assumptions of the proposition, there exists a real-valued continuous function~$f$ defined on~$\sphere{n-1}$ such that
$$\norma{x}^p=\int_{\sphere{n-1}}|\innerproduct{x}{u}|^p\,f(u)\,du\,\quad (x\in\sphere{n-1})$$
See \cite{kol94}, Theorem~1. The conclusion follows now from the uniqueness of the Levy $p$-representation for $p\notin 2\mathbb{N}$, combined with~Theorem~\ref{mainTheorem}.
\end{proof}
{\bf Remark 2}
For the special case $p=1$, Lemma~\ref{theLemma} is stronger than the previous proposition, which for $n=2,p=1$ requires the norm to be at least of class $C^4$.
With a little more work and a few tools from differential and convex geometry,
the $C^2$ assumption in Lemma~\ref{theLemma} can be even further relaxed to the class $C^1$. Here is a sketch of a proof.
The equation  (\ref{1Levy}) can be generalized without any regularity assumptions on $G$, where $G''$ is treated as a distribution in the sense of Schwartz. $G$ being convex guarantees that $G+G''$ is 
a positive measure that must coincide with the $1$-Levy representing measure, by uniqueness of the latter (we are assuming $p=1$). This measure can contain atoms only if the first derivative $G'$ is not a function, but a distribution with a jump-discontinuity. However, such a state of affairs is excluded if $G$ is assumed to be continuously differentiable.

\end{document}